\documentclass{article}%
\usepackage{graphicx}
\usepackage[intlimits]{amsmath}
\usepackage{latexsym}
\usepackage{amsfonts}
\usepackage{amssymb}%
\makeatletter
\newfont{\footsc}{cmcsc10 at 8truept}
\newfont{\footbf}{cmbx10 at 8truept}
\newfont{\footrm}{cmr10 at 10truept}
\newtheorem{theorem}{Theorem}

\newtheorem{corollary}[theorem]{Corollary}

\newtheorem{lemma}[theorem]{Lemma}

\newtheorem{problem}[theorem]{Problem}
\newtheorem{proposition}[theorem]{Proposition}

\newenvironment{proof}[1][Proof]{\noindent{\textbf {#1}  }}  {\hfill$\Box$\bigskip}

\begin{document}

\title{More eigenvalue problems of Nordhaus-Gaddum type}
\author{Vladimir Nikiforov\thanks{Department of Mathematical Sciences, University of
Memphis, Memphis TN 38152, USA; \textit{email: vnikiforv@memphis.edu}} \ and
Xiying Yuan\thanks{Corresponding author. Department of Mathematics, Shanghai
University, Shanghai 200444, China; \textit{email: xiyingyuan2007@hotmail.com
}} \thanks{Research supported by National Science Foundation of China grant
No. 11101263, and by a grant of \textquotedblleft The First-class Discipline
of Universities in Shanghai\textquotedblright.} }
\maketitle

\begin{abstract}
Let $G$ be a graph of order $n$ and let $\mu_{1}\left(  G\right)  \geq
\cdots\geq\mu_{n}\left(  G\right)  $ be the eigenvalues of its adjacency
matrix. This note studies eigenvalue problems of Nordhaus-Gaddum type. Let
$\overline{G}$ be the complement of a graph $G.$ It is shown that if $s\geq2$
and $n\geq15\left(  s-1\right)  ,$ then
\[
\left\vert \mu_{s}\left(  G\right)  \right\vert +|\mu_{s}(\overline{G})|\,\leq
n/\sqrt{2\left(  s-1\right)  }-1.
\]

Also if $s\geq1$ and $n\geq4^{s},$ then
\[
\left\vert \mu_{n-s+1}\left(  G\right)  \right\vert +|\mu_{n-s+1}(\overline
{G})|\,\leq n/\sqrt{2s}+1.
\]
If $s=2^{k}+1$ for some integer $k$, these bounds are asymptotically tight.
These results settle infinitely many cases of a general open problem.\medskip



\end{abstract}

\section{Introduction}

Let $\overline{G}$ denote the complement of a graph $G.$ A
\emph{Nordhaus-Gaddum problem} is of the type:\medskip

\emph{Given a graph parameter }$p\left(  G\right)  ,$\emph{ determine}
\[
\max\left\{  p\left(  G\right)  +p(\overline{G}):v\left(  G\right)
=n\right\}  \text{ \ \ \ or \ \ \ }\min\left\{  p\left(  G\right)
+p(\overline{G}):v\left(  G\right)  =n\right\}  .
\]
Since first introduced by Nordhaus and Gaddum in \cite{NoGa56}, such problems
have been studied for a huge variety of graph parameters; see \cite{AoHa13}
for a recent comprehensive survey. The Nordhaus-Gaddum problems attract
attention because they help to get deeper insights in extremal graph
questions. Also, these problems are the closest analog to Ramsey problems for
non-discrete parameters $p\left(  G\right)  .$

In this note we shall be interested in the case when $p\left(  G\right)  $ is
a spectral graph parameter; thus, given a graph $G$ of order $n,$ let us index
the eigenvalues of the adjacency matrix of $G$ as $\mu_{1}\left(  G\right)
\geq\cdots\geq\mu_{n}\left(  G\right)  $ and set $\mu\left(  G\right)
=\mu_{1}\left(  G\right)  .$

The first known spectral Nordhaus-Gaddum results belong to Nosal \cite{Nos70},
and to Amin and Hakimi \cite{AmHa72}, who showed that for every graph $G$ of
order $n,$
\begin{equation}
n-1\leq\mu\left(  G\right)  +\mu\left(  \overline{G}\right)  <\sqrt{2}\left(
n-1\right)  . \label{Nosin}%
\end{equation}

The lower bound in (\ref{Nosin}) is best possible and is attained if and only
if $G$ is a regular graph; however the upper bound can be improved
significantly. A minor improvement has been shown in \cite{Nik07}, but an
essentially best possible bound has been found only recently, by Csikvari
\cite{Csi09} and Terpai \cite{Ter11} who showed that $\mu\left(  G\right)
+\mu\left(  \overline{G}\right)  \leq4n/3-1.$

A similar problem for other eigenvalues has been proposed in \cite{Nik07}%
:\medskip

\emph{Given }$s$\emph{ and }$n$\emph{, find or estimate the functions}%
\[
f_{s}\left(  n\right)  =\max_{v\left(  G\right)  =n}\left\vert \mu_{s}\left(
G\right)  \right\vert +\left\vert \mu_{s}\left(  \overline{G}\right)
\right\vert \text{ \ \ and \ \ \ }f_{n-s}\left(  n\right)  =\max_{v\left(
G\right)  =n}\left\vert \mu_{n-s+1}\left(  G\right)  \right\vert +\left\vert
\mu_{n-s+1}\left(  \overline{G}\right)  \right\vert .
\]

Several bounds have been proved in \cite{Nik07}; among these is a tight bound
for $f_{2}\left(  n\right)  $:%
\[
\frac{n}{\sqrt{2}}-3<f_{2}\left(  n\right)  <\frac{n}{\sqrt{2}}.
\]

The problem of finding $f_{s}\left(  n\right)  $ for $s\neq2$ has remained
largely open for some time, and has been recently reiterated in \cite{AoHa13}.
In this paper we make further progress along these lines and settle
asymptotically an infinite number of cases. In addition we extend the study to
even more general spectral parameters$.$

Our first statement is about a function similar to $f_{s}\left(  n\right)  $.

\begin{theorem}
\label{thl}If $s\geq2,$ $n\geq3s-2,$ and $G$ is a graph of order $n,$ then
\begin{equation}
\sum_{i=2}^{s}\left(  \mu_{i}^{2}(G)+\mu_{i}^{2}(\overline{G})\right)
<\frac{n^{2}}{4}. \label{in1}%
\end{equation}

\end{theorem}

Applying the AM-QM inequality to (\ref{in1}), we obtain another
Nordhaus-Gaddum result.

\begin{corollary}
\label{cor1}If $s\geq2,$ $n\geq3s-2,$ and $G$ is a graph of order $n,$ then%
\[
\sum_{i=2}^{s}\left(  \left\vert \mu_{i}(G)\right\vert +|\mu_{i}(\overline
{G})|\right)  <n\sqrt{\left(  s-1\right)  /2}.
\]

\end{corollary}

However, we were not able to deduce the following natural statement directly
from Theorem \ref{thl}, so we shall provide a separate proof for it.

\begin{theorem}
\label{thl1}If $s\geq2,$ $n\geq3s-2,$ and $G$ is a graph of order $n,$ then
\begin{equation}
\mu_{s}^{2}(G)+\mu_{s}^{2}(\overline{G})<\frac{n^{2}}{4\left(  s-1\right)  }.
\label{in4}%
\end{equation}

\end{theorem}

Applying the AM-QM inequality to the left side of (\ref{in4}), one immediately
sees that
\[
\left\vert \mu_{s}\left(  G\right)  \right\vert +|\mu_{s}(\overline
{G})|\,<\frac{n}{\sqrt{2\left(  s-1\right)  }},
\]
which is a new bound on $f_{s}\left(  n\right)  .$ However, we can do better,
using a trick that was pioneered by the first author in \cite{Nik06a}, and has
been applied on numerous occasions since then. We thus get the following bound.

\begin{theorem}
\label{cor2}If $s\geq2,$ $n\geq15\left(  s-1\right)  ,$ and $G$ is a graph of
order $n,$ then
\[
\left\vert \mu_{s}\left(  G\right)  \right\vert +|\mu_{s}(\overline{G}%
)|\,\leq\frac{n}{\sqrt{2\left(  s-1\right)  }}-1.
\]

\end{theorem}

It turns out that the last inequality is asymptotically tight, at least for
some values of $s,$ as shown in Theorem \ref{thlb} below.

Finding $f_{n-s}\left(  n\right)  $ turns to be slightly different. We begin
with an analog of Theorem \ref{thl}.

\begin{theorem}
\label{ths}If $s\geq1,$ $n>2s,$ and $G$ is a graph of order $n,$ then
\begin{equation}%
{\displaystyle\sum\limits_{i=1}^{s}}
\left(  \mu_{n-i+1}^{2}\left(  G\right)  +\mu_{n-i+1}^{2}(\overline
{G})\right)  \leq\left(  \frac{n}{2}+s\right)  ^{2}. \label{in3}%
\end{equation}

\end{theorem}

From (\ref{in3}) we easily obtain another Nordhaus-Gaddum result, similar to
Corollary \ref{cor1}.

\begin{corollary}
If $s\geq1,$ $n>2s,$ and $G$ is a graph of order $n,$ then%
\[%
{\displaystyle\sum\limits_{i=1}^{s}}
\left(  \left\vert \mu_{n-i+1}\left(  G\right)  \right\vert +|\mu
_{n-i+1}(\overline{G})|\right)  \leq\left(  \frac{n}{2}+s\right)  \sqrt{2s}.
\]

\end{corollary}

We also can deduce the following corollary, whose short proof is in Section
\ref{pf}.

\begin{corollary}
\label{cor4}If $s\geq1,$ $n>4^{s},$ and $G$ is a graph of order $n,$ then
\begin{equation}
\mu_{n-s+1}^{2}\left(  G\right)  +\mu_{n-s+1}^{2}(\overline{G})\leq\frac{1}%
{s}\left(  \frac{n}{2}+s\right)  ^{2}. \label{i5}%
\end{equation}

\end{corollary}

Note that the right side of (\ref{i5}) includes low order terms. Such terms
may be reduced but not removed completely, at least for some values of $s:$
e.g., if $s=1,$ taking the complete balanced bipartite graph $K_{n/2,n/2},$ we
see that
\[
\mu_{n-s+1}^{2}\left(  K_{n/2,n/2}\right)  +\mu_{n-s+1}^{2}(\overline
{K_{n/2,n/2}})=\frac{n^{2}}{4}+1.
\]

We shall use Corollary \ref{cor4} to obtain a new bound on $f_{n-s}\left(
n\right)  $ as well.

\begin{theorem}
\label{ths1}If $s\geq1,$ $n\geq4^{s},$ and $G$ is a graph of order $n,$ then
\[
\left\vert \mu_{n-s+1}\left(  G\right)  \right\vert +|\mu_{n-s+1}(\overline
{G})|\,\leq\frac{n}{\sqrt{2s}}+1.
\]

\end{theorem}

All above bounds are essentially best possible whenever $s=2^{k}+1$ and $n$ is
sufficiently large.

\begin{theorem}
\label{thlb}Let $s=2^{k-1}+1$ for some integer $k\geq1.$ There exists
infinitely many graphs $G$ such that if $2\leq i\leq s,$ then%
\begin{align*}
\mu_{i}\left(  G\right)   &  \geq\frac{v\left(  G\right)  }{2\sqrt{2\left(
s-1\right)  }}-1,\text{ \ \ \ }\mu_{n-i+2}\left(  G\right)  \leq
-\frac{v\left(  G\right)  }{2\sqrt{2\left(  s-1\right)  }},\\
\mu_{i}\left(  \overline{G}\right)   &  \geq\frac{v\left(  G\right)  }%
{2\sqrt{2\left(  s-1\right)  }}-1,\text{ \ \ \ }\mu_{n-i+2}\left(
\overline{G}\right)  \leq-\frac{v\left(  G\right)  }{2\sqrt{2\left(
s-1\right)  }}.
\end{align*}

\end{theorem}

\section{\label{pf}Proofs}

For graph notation and concepts undefined here, we refer the reader to
\cite{Bol98}. In particular, if $G$ is a graph, we write $v\left(  G\right)  $
for the number of vertices of $G.$ For short we set $\mu_{i}:=\mu_{i}\left(
G\right)  ,$ $\mu:=\mu\left(  G\right)  ,$ $\overline{\mu}_{i}:=\mu_{i}\left(
\overline{G}\right)  ,$ and $\overline{\mu}:=\mu\left(  \overline{G}\right)
.$

\subsection{Some useful observations}

\begin{lemma}
\label{a1}Let $G$ be a graph of order $n.$ If $X\subset\left\{  2,\ldots
,n\right\}  ,$ then%
\[
\sum_{i\in X}\mu_{i}^{2}\left(  G\right)  \leq n^{2}/4.
\]

\end{lemma}

\begin{proof}
Indeed, if $A$ is the adjacency matrix of $G,$ then
\[
\mu^{2}+\sum_{i\in X}\mu_{i}^{2}\leq\sum_{i=1}^{n}\mu_{i}^{2}=\mathrm{tr}%
\left(  A^{2}\right)  =2e\left(  G\right)  .
\]
Hence, in view of $\mu\geq2e\left(  G\right)  /n,$ we find that
\[
\sum_{i\in X}\mu_{i}^{2}\leq2e\left(  G\right)  -\mu^{2}\leq2e\left(
G\right)  -\left(  2e\left(  G\right)  /n\right)  ^{2}\leq n^{2}/4,
\]
completing the proof.
\end{proof}

\begin{lemma}
\label{a2}Let $n\geq s\geq2,$ and let $G$ be a graph of order $n.$ If $\mu
_{s}\leq0,$ then%
\[
\left\vert \mu_{s}\right\vert \leq\frac{n}{2\sqrt{n-s+1}}.
\]

\end{lemma}

\begin{proof}
Indeed, since $\mu_{n}\left(  G\right)  \leq\cdots\leq\mu_{s}\left(  G\right)
\leq0,$ we see that
\[
\sum_{i=s}^{n}\mu_{i}^{2}\geq\sum_{i=s}^{n}\mu_{s}^{2}=\left(  n-s+1\right)
\mu_{s}^{2},
\]
and the assertion follows by Lemma \ref{a1}.
\end{proof}

\begin{theorem}
\label{tR}Let $k\geq0$ and $n\geq4^{k}.$ If $G$ is a graph of order $n,$ then
either
\[
\mu_{n-k+1}\left(  G\right)  \leq-1\text{ \ \ and \ \ }\mu_{n-k+1}%
(\overline{G})\leq0
\]
or%
\[
\mu_{n-k+1}(\overline{G})\leq-1\text{ \ \ and \ \ }\mu_{n-k+1}\left(
G\right)  \leq0.
\]

\end{theorem}

\begin{proof}
A classical bound of Ramsey theory implies that every graph of order at least
$4^{k}$ contains either a complete graph on $k+1$ vertices or an independent
set on $k+1$ vertices. Suppose that $G$ contains a complete graph on $k+1$
vertices, and so $\overline{G}$ contains an independent set on $k+1$ vertices.
For an induced subgraph $H$ of graph $G,$ the Cauchy interlacing theorem
implies that
\[
\mu_{m-i}\left(  H\right)  \geq\mu_{n-i}\left(  G\right)
\]
for all $i=0,\ldots,v\left(  H\right)  -1;$ therefore,
\[
\mu_{n-k+1}\left(  G\right)  \leq\mu_{2}\left(  K_{k+1}\right)  =-1\text{
\ \ and \ \ }\mu_{n-k+1}(\overline{G})\leq\mu_{2}\left(  \overline{K_{k+1}%
}\right)  =0
\]
as claimed.
\end{proof}

Using Weyl's inequalities (\cite{HoJo88}, p. 181), we come up with the
following pair of useful bounds:

\begin{lemma}
[Weyl]If $G$ is a graph of order $n$ and $2\leq k\leq n,$ then%
\begin{align}
\mu_{k}\left(  G\right)  +\mu_{n-k+2}(\overline{G})  &  \leq-1,\label{wi1}\\
\mu_{k}\left(  G\right)  +\mu_{n-k+1}(\overline{G})  &  \geq-1. \label{wi2}%
\end{align}

\end{lemma}

\subsection{Blow-ups of graphs}

For any graph $G$ and integer $t\geq1,$ write $G^{\left(  t\right)  }$ for the
graph obtained by replacing each vertex $u$ of $G$ by a set $V_{u}$ of $t$
independent vertices and every edge $\left\{  u,v\right\}  $ of $G$ by a
complete bipartite graph with parts $V_{u}$ and $V_{v}.$ Usually $G^{\left(
t\right)  }$ is called a blow-up of $G.$ Blow-up graphs have a very useful
algebraic relation to $G:$ thus, if $A$ is the adjacency matrix of $G,$ then
the adjacency matrix $A\left(  G^{\left(  t\right)  }\right)  $ of $G^{\left(
t\right)  }$ is given by
\[
A\left(  G^{\left(  t\right)  }\right)  =A\otimes J_{t}%
\]
where $\otimes$ is the Kronecker product and $J_{t}$ is the all ones matrix of
order $t.$ This observation yields the following fact.

\begin{proposition}
\label{pro1}The\ eigenvalues of $G^{\left(  t\right)  }$ are $t\mu_{1}\left(
G\right)  ,\ldots,t\mu_{n}\left(  G\right)  ,$ together with $n\left(
t-1\right)  $ additional $0$'s.
\end{proposition}

We also want to find the eigenvalues of the complements of graph blow-ups.
Given a graph $G$ and an integer $t>0,$ set $G^{\left[  t\right]  }%
=\overline{\overline{G}^{\left(  t\right)  }},$ i.e., $G^{\left[  t\right]  }$
is obtained from $G^{\left(  t\right)  }$ by joining all vertices within
$V_{u}$ for every vertex $u$ of $G.$ We easily can check the following fact.

\begin{proposition}
\label{pro2}The\ eigenvalues of $G^{\left[  t\right]  }$ are $t\mu_{1}\left(
G\right)  +t-1,\ldots,t\mu_{n}\left(  G\right)  +t-1,$ together with $n\left(
t-1\right)  $ additional $\left(  -1\right)  $'s.
\end{proposition}

\subsection{Proofs}

\begin{proof}
[\textbf{Proof of Theorem \ref{thl}}]Let $2\leq i\leq s.$ First, we shall show
that
\begin{equation}
\mu_{i}^{2}\leq\mu_{s+i-1}^{2}+\overline{\mu}_{n-i+2}^{2}. \label{in11}%
\end{equation}
Indeed, if $\mu_{i}\geq0,$ then (\ref{wi1}) implies that $\overline{\mu
}_{n-i+2}\leq0,$ and so
\[
\mu_{i}^{2}<\left(  \mu_{i}+1\right)  ^{2}\leq\overline{\mu}_{n-i+2}^{2}.
\]
On the other hand, if $\mu_{i}<0,$ then $\mu_{s+i-1}\leq\mu_{i}<0$ implies
that $\mu_{i}^{2}\leq\mu_{s+i-1}^{2}.$ So (\ref{in11}) is always true.

Further, we obviously have%
\[
\mu^{2}+\sum_{i=2}^{n}\mu_{i}^{2}+\overline{\mu}^{2}+\sum_{i=2}^{n}%
\overline{\mu}_{i}^{2}=2e\left(  G\right)  +2e\left(  \overline{G}\right)
=n(n-1).
\]
Note that Weyl's inequality (\ref{in11}) implies that%
\begin{equation}
2\sum_{i=2}^{s}\mu_{i}^{2}\leq\sum_{i=2}^{s}\mu_{i}^{2}+\sum_{i=s+1}^{2s-1}%
\mu_{i}^{2}+\sum_{i=n-s+2}^{n}\overline{\mu}_{i}^{2}, \label{i1}%
\end{equation}
and by symmetry,%
\begin{equation}
2\sum_{i=2}^{s}\overline{\mu}_{i}^{2}\leq\sum_{i=2}^{s}\overline{\mu}_{i}%
^{2}+\sum_{i=s+1}^{2s-1}\overline{\mu}_{i}^{2}+\sum_{i=n-s+2}^{n}\mu_{i}^{2}.
\label{i2}%
\end{equation}
Further, the condition $n\geq3s-2,$ implies that $n-s+2>2s-1$ and so
\[
\left\{  s+1,\ldots,2s-1\right\}  \cap\left\{  n-s+2,\ldots,n\right\}
=\varnothing.
\]
Therefore, adding (\ref{i1}) and (\ref{i2}) together with $\mu^{2}$ and
$\overline{\mu}^{2},$ we see that
\[
\mu^{2}+2\sum_{i=2}^{s}\mu_{i}^{2}+\overline{\mu}^{2}+2\sum_{i=2}^{s}%
\overline{\mu}_{i}^{2}\leq\sum_{i=1}^{n}\mu_{i}^{2}+\sum_{i=1}^{n}%
\overline{\mu}_{i}^{2}\leq n\left(  n-1\right)  .
\]
Finally, using (\ref{Nosin}) we find that $\mu^{2}+\overline{\mu}^{2}%
\geq\left(  \mu+\overline{\mu}\right)  ^{2}/2\geq(n-1)^{2}/2,$ and so%
\[
\sum_{i=2}^{s}\mu_{i}^{2}+\sum_{i=2}^{s}\overline{\mu}_{i}^{2}\leq\frac{1}%
{2}n\left(  n-1\right)  -\frac{1}{4}(n-1)^{2}=\frac{n^{2}-1}{4}<\frac{n^{2}%
}{4},
\]
completing the proof of Theorem \ref{thl}.
\end{proof}

\bigskip

\begin{proof}
[\textbf{Proof of Theorem \ref{thl1}}]If $\mu_{s}\left(  G\right)  \geq0$ and
$\mu_{s}(\overline{G})\geq0,$ then
\[
\sum_{i=2}^{s}\left(  \mu_{i}^{2}+\overline{\mu}_{i}^{2}\right)  \geq\left(
s-1\right)  \left(  \mu_{s}^{2}+\overline{\mu}_{s}^{2}\right)  ,
\]
and inequality (\ref{in4}) follows by Theorem \ref{thl}.

Next, if $\mu_{s}<0$ and $\overline{\mu}_{s}<0,$ then Lemma \ref{a2} implies
that
\[
\mu_{s}^{2}+\overline{\mu}_{s}^{2}\leq\frac{n^{2}}{2\left(  n-s+1\right)
}<\frac{n^{2}}{4\left(  s-1\right)  },
\]
so (\ref{in4}) follows in this case as well.

Finally, assume that $\mu_{s}<0$ and $\overline{\mu}_{s}\geq0.$ Then
$\mu_{2s-1}\leq\cdots\leq\mu_{s+1}\leq\mu_{s}<0$ and so,
\begin{equation}
\left(  s-1\right)  \mu_{s}^{2}\leq\mu_{s+1}^{2}+\cdots+\mu_{2s-1}^{2}.
\label{i3}%
\end{equation}
Since $\overline{\mu}_{2}\geq\cdots\geq\overline{\mu}_{s}\geq0,$ inequality
(\ref{wi1}) implies that $\overline{\mu}_{n-i+2}\leq0,$ and so%
\[
\overline{\mu}_{s}^{2}\leq\overline{\mu}_{i}^{2}<\left(  \overline{\mu}%
_{i}+1\right)  ^{2}\leq\mu_{n-i+2}^{2}%
\]
for every $i=2,\ldots,s.$ Therefore
\begin{equation}
\left(  s-1\right)  \overline{\mu}_{s}^{2}\leq\mu_{n}^{2}+\cdots+\mu
_{n-s+2}^{2}. \label{i4}%
\end{equation}
Since the condition $n\geq3s-2,$ implies that $n-s+2>2s-1,$ we see that
\[
\left\{  s+1,\ldots,2s-1\right\}  \cap\left\{  n-s+2,\ldots,n\right\}
=\varnothing.
\]
Hence, setting $X:=\left\{  s+1,\ldots,2s-1\right\}  \cup\left\{
n-s+2,\ldots,n\right\}  ,$ inequalities (\ref{i3}), (\ref{i4}) and Lemma
\ref{a1} imply that
\[
\left(  s-1\right)  \left(  \mu_{s}^{2}+\overline{\mu}_{s}^{2}\right)
\leq\sum_{i=s+1}^{2s-1}\mu_{i}^{2}+\sum_{i=n-s+2}^{n}\mu_{i}^{2}=\sum_{i\in
X}\mu_{i}^{2}\leq\frac{n^{2}}{4},
\]
completing the proof of Theorem \ref{thl1}.
\end{proof}

\bigskip

\begin{proof}
[\textbf{Proof of Theorem \ref{cor2}}]Assume that $s\geq2,$ $n\geq15\left(
s-1\right)  $, and $G$ is a graph of order $n.$ As mentioned above, using the
AM - QM inequality and Theorem \ref{thl1}, we always have
\begin{equation}
\left\vert \mu_{s}\left(  G\right)  \right\vert +|\mu_{s}(\overline{G})|\text{
}\leq\text{ }\frac{n}{\sqrt{2\left(  s-1\right)  }}. \label{ub}%
\end{equation}
To the end of the proof we shall show that we can add a $-1$ to the right side
of this inequality. Thus, let $G$ be a graph of order $n$ with
\[
\left\vert \mu_{s}\left(  G\right)  \right\vert +|\mu_{s}(\overline{G})|\text{
}=\text{ }f_{s}\left(  n\right)  ,
\]
and assume for a contradiction that
\begin{equation}
f_{s}\left(  n\right)  >\text{ }\frac{n}{\sqrt{2\left(  s-1\right)  }}-1,
\label{as0}%
\end{equation}
Our first aim is to show that $\mu_{s}>0$ and $\overline{\mu}_{s}>0.$ Indeed,
if both $\mu_{s}$ and $\overline{\mu}_{s}$ are non-positive, then Lemma
\ref{a2} implies that
\[
\left\vert \mu_{s}\left(  G\right)  \right\vert +|\mu_{s}(\overline{G})|\text{
}<\frac{n}{\sqrt{n-s+1}}\leq\frac{n}{\sqrt{2\left(  s-1\right)  }}-1.
\]
Now, let $\mu_{s}>0$ and $\overline{\mu}_{s}\leq0$. Then, Lemmas \ref{a1} and
\ref{a2} imply that
\[
\left\vert \mu_{s}\right\vert \leq\frac{n}{2\sqrt{s-1}}\text{ \ \ and
\ \ }|\overline{\mu}_{s}|\text{ }\leq\frac{n}{2\sqrt{n-s+1}},
\]
and so, in view of $n\geq15\left(  s-1\right)  ,$ we see that
\[
\frac{n}{2\sqrt{s-1}}+\frac{n}{2\sqrt{n-s+1}}<\frac{n}{\sqrt{2\left(
s-1\right)  }}-1,
\]
contradicting (\ref{as0}).\ Therefore $\mu_{s}>0$ and $\overline{\mu}_{s}>0.$

Now, let $t$ be a positive integer and set $H:=G^{\left(  t\right)  }.$ Since
$\mu_{s}>0$ and $\overline{\mu}_{s}>0,$ Propositions \ref{pro1} and \ref{pro2}
imply that
\[
\mu_{s}\left(  H\right)  =t\mu_{s}\text{ \ \ and \ \ }\mu_{s}\left(
\overline{H}\right)  =\mu_{s}(\overline{G}^{\left[  t\right]  })=t\overline
{\mu}_{s}+t-1,
\]
and therefore
\[
\left\vert \mu_{s}\left(  H\right)  \right\vert +|\mu_{s}\left(  \overline
{H}\right)  |\text{ }=tf_{s}\left(  n\right)  +t-1.
\]
On the other hand, (\ref{ub}) implies that
\[
\left\vert \mu_{s}\left(  H\right)  \right\vert +|\mu_{s}\left(  \overline
{H}\right)  |\text{ }\leq\frac{tn}{\sqrt{2\left(  s-1\right)  }};
\]
hence%
\[
f_{s}\left(  n\right)  \leq\frac{n}{\sqrt{2\left(  s-1\right)  }}-\frac
{t-1}{t}.
\]
Now, letting $t$ tend to $\infty,$ we obtain a contradiction to (\ref{as0}),
and thus complete the proof of Theorem \ref{cor2}.
\end{proof}

\bigskip

\begin{proof}
[\textbf{Proof of Theorem \ref{ths}}]We start with the obvious fact
\begin{equation}%
{\displaystyle\sum\limits_{i=2}^{n}}
\left(  \mu_{i}+\overline{\mu}_{i}\right)  =-\mu-\overline{\mu}. \label{eq1}%
\end{equation}
For $i=2,\ldots,n,$ set $w_{i}:=\mu_{i}+\overline{\mu}_{n-i+2}.$ Rearranging
(\ref{eq1}) and using (\ref{wi1}), we find that
\[%
{\displaystyle\sum\limits_{i=2}^{s+1}}
(w_{i}+w_{n-i+2})=-\mu-\overline{\mu}-%
{\displaystyle\sum\limits_{i=s+2}^{n-s}}
w_{i}\geq-\mu-\overline{\mu}+n-2s-1.
\]

On the other hand,
\[
\mu_{i}^{2}=\left(  \overline{\mu}_{n-i+2}-w_{i}\right)  ^{2}=\overline{\mu
}_{n-i+2}^{2}-2w_{i}\overline{\mu}_{n-i+2}+w_{i}^{2}>\overline{\mu}%
_{n-i+2}^{2}-2w_{i}\overline{\mu}_{n-i+2}.
\]
Since $w_{i}<0$ and $\overline{\mu}_{n-i+2}\geq-n/2$ (by Lemma \ref{a2})$,$ we
see that%
\[
\mu_{i}^{2}\geq\overline{\mu}_{n-i+2}^{2}+nw_{i}.
\]
Therefore,%
\begin{align*}%
{\displaystyle\sum\limits_{i=2}^{s+1}}
\left(  \mu_{i}^{2}+\overline{\mu}_{i}^{2}\right)   &  \geq%
{\displaystyle\sum\limits_{i=2}^{s+1}}
\left(  \mu_{n-i+2}^{2}+\overline{\mu}_{n-i+2}^{2}\right)  +n%
{\displaystyle\sum\limits_{i=2}^{s+1}}
\left(  w_{i}+w_{n-i+2}\right) \\
&  =%
{\displaystyle\sum\limits_{i=1}^{s}}
\left(  \mu_{n-i+1}^{2}+\overline{\mu}_{n-i+1}^{2}\right)  +n%
{\displaystyle\sum\limits_{i=2}^{s+1}}
\left(  w_{i}+w_{n-i+2}\right) \\
&  \geq%
{\displaystyle\sum\limits_{i=1}^{s}}
\left(  \mu_{n-i+1}^{2}+\overline{\mu}_{n-i+1}^{2}\right)  +n\left(
n-2s-1-\mu-\overline{\mu}\right)  .
\end{align*}

Using that $n>2s,$ we see that
\[
\mu^{2}+\overline{\mu}^{2}+%
{\displaystyle\sum\limits_{i=2}^{s+1}}
(\mu_{i}^{2}+\overline{\mu}_{i}^{2})+%
{\displaystyle\sum\limits_{i=1}^{s}}
(\mu_{n-i+1}^{2}+\overline{\mu}_{n-i+1}^{2})\leq n\left(  n-1\right)  ,
\]
and so,
\[
\mu^{2}+\overline{\mu}^{2}+2%
{\displaystyle\sum\limits_{i=1}^{s}}
(\mu_{n-i+1}^{2}+\overline{\mu}_{n-i+1}^{2})+n\left(  n-2s-1-\mu-\overline
{\mu}\right)  \leq n\left(  n-1\right)  .
\]
Rearranging this inequality, we find that
\begin{align*}
2%
{\displaystyle\sum\limits_{i=1}^{s}}
(\mu_{n-i+1}^{2}+\overline{\mu}_{n-i+1}^{2})  &  \leq2sn+n\left(
\mu+\overline{\mu}\right)  -\mu^{2}-\overline{\mu}^{2}\\
&  \leq2sn+n\left(  \mu+\overline{\mu}\right)  -\left(  \mu+\overline{\mu
}\right)  ^{2}/2\\
&  \leq2sn+n^{2}/2,
\end{align*}
completing the proof of Theorem \ref{ths}.
\end{proof}

\bigskip

\begin{proof}
[\textbf{Proof of Corollary \ref{cor4}}]Since Theorem \ref{tR} implies that
$\mu_{n-s+1}\leq0$ and $\overline{\mu}_{n-s+1}\leq0,$ we get%
\[%
{\displaystyle\sum\limits_{i=1}^{s}}
\left(  \mu_{n-i+1}^{2}+\overline{\mu}_{n-i+1}^{2}\right)  \geq s\left(
\mu_{n-s+1}^{2}+\overline{\mu}_{n-s+1}^{2}\right)  ,
\]
and the assertion follows by Theorem \ref{ths}.
\end{proof}

\bigskip

\begin{proof}
[\textbf{Proof of Theorem \ref{ths1}}]Assume that $s\geq1,$ $n\geq4^{s}$, and
$G$ is a graph of order $n$ with
\[
\left\vert \mu_{n-s+1}\left(  G\right)  \right\vert +|\mu_{n-s+1}(\overline
{G})|\text{ }=\text{ \ }f_{n-s}\left(  n\right)  .
\]
To begin with, using the AM - QM inequality and Corollary \ref{cor4}, we see
that
\[
f_{n-s}\left(  n\right)  \leq\text{ }\frac{n}{\sqrt{2s}}+\sqrt{2s}.
\]
Note that this inequality is almost what we need, as the main term is the
correct one, but the constant term is larger than desired. Thus, to the end of
the proof we shall show that we can make the additive term equal to $1$. We
shall use the same techniques as in the proof of Theorem \ref{cor2}.

Assume for a contradiction that
\begin{equation}
f_{n-s}\left(  n\right)  >\text{ }\frac{n}{\sqrt{2s}}+1, \label{ass}%
\end{equation}
First, Theorem \ref{tR} implies that $\overline{\mu}_{n-s+1}\leq0$, and so
Lemma \ref{a2} implies that $\left\vert \overline{\mu}_{n-s+1}\right\vert \leq
n/\left(  2\sqrt{s}\right)  .$ Hence,
\[
\left\vert \mu_{n-s+1}\right\vert >\frac{n}{\sqrt{2s}}+1-\left\vert
\overline{\mu}_{n-s+1}\right\vert >\frac{n}{\sqrt{2s}}+1-\frac{n}{2\sqrt{s}%
}>1,
\]
and, by symmetry, $\left\vert \overline{\mu}_{n-s+1}\right\vert >1.$ That is
to say,
\[
\mu_{n-s+1}<-1\text{ \ \ and \ }\overline{\mu}_{n-s+1}<-1.
\]

Let $t$ be a positive integer and set $H:=G^{\left(  t\right)  }.$ Since
$\mu_{n-s+1}<-1$ and $\overline{\mu}_{n-s+1}<-1,$ Propositions \ref{pro1} and
\ref{pro2} imply that
\[
\mu_{tn-s+1}\left(  H\right)  =t\mu_{n-s+1}\text{ \ \ and \ \ }\mu
_{tn-s+1}\left(  \overline{H}\right)  =\mu_{tn-s+1}(\overline{G}^{\left[
t\right]  })=t\overline{\mu}_{n-s+1}+t-1,
\]
and therefore
\[
\left\vert \mu_{tn-s+1}\left(  H\right)  \right\vert +|\mu_{tn-s+1}\left(
\overline{H}\right)  |\text{ }=tf_{n-s}\left(  n\right)  -t+1.
\]
On the other hand, Corollary \ref{cor4} implies that
\[
\left\vert \mu_{tn-s+1}\left(  H\right)  \right\vert +|\mu_{tn-s+1}%
(\overline{H})|\text{ }\leq\frac{tn}{\sqrt{2s}}+\sqrt{2s};
\]
hence%
\[
f_{n-s}\left(  n\right)  \leq\frac{n}{\sqrt{2s}}+\frac{\sqrt{2s}}{t}%
+\frac{t-1}{t}.
\]
Now, letting $t$ tend to $\infty,$ we obtain a contradiction to (\ref{ass}),
and thus complete the proof of Theorem \ref{ths1}.
\end{proof}

\section{Lower bounds on $f_{s}\left(  n\right)  $ and $f_{n-s}\left(
n\right)  $}

Define an infinite sequence of square $\left(  0,1\right)  $ matrices
$A_{1},A_{2},\ldots$ as follows. Let
\[
A_{1}=\left[
\begin{array}
[c]{cc}%
1 & 0\\
0 & 1
\end{array}
\right]  ,\text{ \ }B=\left[
\begin{array}
[c]{cc}%
1 & -1\\
-1 & -1
\end{array}
\right]  ,
\]
and for every $k=1,2,\ldots$ set%
\[
A_{k+1}=\frac{1}{2}\left(  (2A_{k}-J_{2^{k}})\otimes B+J_{2^{k+1}}\right)  .
\]
First note that $A_{k+1}$ is a $(0,1)$ symmetric matrix of order $2^{k+1}.$

To give some properties of the eigenvalues of the matrices $A_{k+1}$ we first
point out a fact without a proof.

\begin{lemma}
\label{le3}Let $M$ be a symmetric real matrix of order $n$ with all row-sums
equal to $r,$ and $r,\mu_{2}(M),\ldots,\mu_{n}(M)$ be the eigenvalues of $M.$
If $a$ and $b$ are real numbers, then the eigenvalues of the matrix
$aM+bJ_{n}$ are
\[
ar+bn,a\mu_{2}(M),\ldots,a\mu_{n}(M).
\]

\end{lemma}

In the following lemma and its proof we shall use $a^{\left[  b\right]  }$ to
indicate an eigenvalue $a$ of multiplicity $b.$

\begin{lemma}
\label{le4}If $k\geq1,$ then the row-sums of $A_{k+1}$ are equal to $2^{k}$
and its spectrum is
\[
2^{k},\left(  2^{k/2}\right)  ^{\left[  2^{k-1}\right]  },0^{\left[
2^{k}-1\right]  },\left(  -2^{k/2}\right)  ^{\left[  2^{k-1}\right]  }.
\]

\end{lemma}

\begin{proof}
We shall prove the lemma by induction on $k$. If $k=1,$ we see that
\[
A_{2}=\frac{1}{2}\left(  (2A_{1}-J_{2})\otimes B+J_{4}\right)  =\left[
\begin{array}
[c]{cccc}%
1 & 0 & 0 & 1\\
0 & 0 & 1 & 1\\
0 & 1 & 1 & 0\\
1 & 1 & 0 & 0
\end{array}
\right]  .
\]
The row-sums of $A_{2}$ are equal to $2$, and its eigenvalues are $2,$
$\sqrt{2},0$ and $-\sqrt{2}.$ Assume that the statement holds for $A_{k};$ in
particular, the row-sums of $A_{k}$ are equal to $2^{k-1}.$ Hence, the
row-sums of both $2A_{k}-J_{2^{k}}\ $and $(2A_{k}-J_{2^{k}})\otimes B$ are
zero, and so the row-sums of $A_{k+1}$ are $2^{k},$ proving the first part of
the statement.

Further, the spectrum of $A_{k}$ is
\[
2^{k-1},\text{ }\left(  2^{\left(  k-1\right)  /2}\right)  ^{\left[
2^{k-2}\right]  },\text{ }0^{\left[  2^{k-1}-1\right]  },\text{ }\left(
-2^{\left(  k-1\right)  /2}\right)  ^{\left[  2^{k-2}\right]  }.
\]
Thus, by Lemma \ref{le3} the spectrum of $2A_{k}-J_{2^{k}}$ is%
\[
\left(  2^{\left(  k+1\right)  /2}\right)  ^{\left[  2^{k-2}\right]  },\text{
}0^{\left[  2^{k-1}\right]  },\text{ }\left(  -2^{\left(  k+1\right)
/2}\right)  ^{\left[  2^{k-2}\right]  }.
\]
Since the eigenvalues of $B$ are $\sqrt{2}$ and $-\sqrt{2},$ the spectrum of
$(2A_{k}-J_{2^{k}})\otimes B$ is%
\[
\left(  2^{\left(  k+2\right)  /2}\right)  ^{\left[  2^{k-1}\right]  },\text{
}0^{\left[  2^{k}\right]  },\text{ }\left(  -2^{\left(  k+2\right)
/2}\right)  ^{\left[  2^{k-1}\right]  }.
\]
Finally, the row-sums of $(2A_{k}-J_{2^{k}})\otimes B$ are zero, so by Lemma
\ref{le3} the spectrum of $A_{k+1}$ is
\[
2^{k},\text{ }\left(  2^{k/2}\right)  ^{\left[  2^{k-1}\right]  },\text{
}0^{\left[  2^{k}-1\right]  },\text{ }\left(  -2^{k/2}\right)  ^{\left[
2^{k-1}\right]  },
\]
completing the induction step and the proof of Lemma \ref{le4}.
\end{proof}

If $P$ is a Hermitian matrix of size $n,$ we index the eigenvalues of $P$ as
$\mu_{1}\left(  P\right)  \geq\cdots\geq\mu_{n}\left(  P\right)  .$ Observe
the following special case of Weyl's inequalities that we shall need in the
proof of Theorem \ref{thlb}.

If $P$ and $Q$ are Hermitian matrices of size $n,$ then for each $1\leq s\leq
n$
\[
\mu_{n}(P-Q)\leq\mu_{s}(P)-\mu_{s}(Q)\leq\mu_{1}(P-Q).
\]

\begin{proof}
[\textbf{Proof of Theorem \ref{thlb}}]Let $A_{k+1}(t)$ be matrix obtained from
$A_{k+1}\otimes J_{t}$ by zeroing all diagonal entries. Clearly $A_{k+1}(t)$
is the adjacency matrix of a graph $G$ of order $n=2^{k+1}t.$ For
$s=2^{k-1}+1,$ and for each $i\leq s$ we have
\[
\mu_{i}\left(  G\right)  \geq\mu_{s}\left(  G\right)  \geq\mu_{s}%
(A_{k+1}\otimes J_{t})-\mu_{1}\left(  A_{k+1}\otimes J_{t}-A_{k+1}(t)\right)
=2^{k/2}t-1=\frac{v\left(  G\right)  }{2\sqrt{2\left(  s-1\right)  }}-1.
\]
Next, it is not hard to see that the adjacency matrix of $\overline{G}$ is
obtained from $\left(  J_{2^{k+1}}-A_{k+1}\right)  \otimes J_{t}$ by zeroing
all diagonal entries. Since $A_{k+1}$ and $J_{2^{k+1}}-A_{k+1}$ have the same
spectrum, we also find that
\[
\mu_{i}\left(  \overline{G}\right)  \geq\mu_{s}\left(  \overline{G}\right)
\geq\mu_{s}(\left(  J_{2^{k+1}}-A_{k+1}\right)  \otimes J_{t})-1=2^{k/2}%
t-1=\frac{v\left(  G\right)  }{2\sqrt{2\left(  s-1\right)  }}-1.
\]
Finally, we have
\[
\mu_{n-i+2}\left(  G\right)  \leq\mu_{n-s+2}\left(  G\right)  \leq\mu
_{n-s+2}(A_{k+1}\otimes J_{t})-\mu_{n}\left(  A_{k+1}\otimes J_{t}%
-A_{k+1}(t)\right)  =-2^{k/2}t=-\frac{v\left(  G\right)  }{2\sqrt{2\left(
s-1\right)  }},
\]
and the same bound holds also for $\mu_{n-i+2}\left(  \overline{G}\right)  .$
The proof of Theorem \ref{thlb} is completed.
\end{proof}

\section{Concluding remarks}

We would like to emphasize the decisive role of Weyl's inequalities in our
proofs. It turns out that they offer almost unlimited possibilities for
variations. The upper bounds on $f_{s}\left(  n\right)  $ determined in
Corollary \ref{cor2} and Theorem \ref{ths} seem asymptotically tight for every
$s$ and $n$ tending to infinity. However, if would be hard to disprove such
conjecture if it turns out to be false. Thus, we raise the following problem.

\begin{problem}
For which values of $s$ it is true that%
\begin{equation}
\lim_{n\rightarrow\infty}\frac{1}{n}f_{s}\left(  n\right)  =\frac{1}%
{\sqrt{2\left(  s-1\right)  }}? \label{l1}%
\end{equation}
For which values of $s$ it is true that
\begin{equation}
\lim_{n\rightarrow\infty}\frac{1}{n}f_{n-s}\left(  n\right)  =\frac{1}%
{\sqrt{2s}}? \label{l2}%
\end{equation}

\end{problem}

Further, for those $s$ for which the answer to of the above problem is
positive, we can ask subtler and more definite questions.

\begin{problem}
If $s$ is such that equality holds in (\ref{l1}), is it true that
\[
\lim_{n\rightarrow\infty}\left(  f_{s}\left(  n\right)  \text{ }-\frac
{n}{\sqrt{2\left(  s-1\right)  }}\right)  =-1?
\]
If $s$ is such that equality holds in (\ref{l2}), is it true that
\[
\lim_{n\rightarrow\infty}\left(  \frac{1}{n}f_{n-s}\left(  n\right)  -\frac
{1}{\sqrt{2s}}\right)  =1?
\]

\end{problem}

\bigskip

\textbf{Acknowledgements.} Part of this work was done while the second author
was visiting the University of Memphis in 2012/2013. She is grateful for the
hospitality of the University during her stay.


\begin{thebibliography}{99}                                                                                               %


\bibitem {AmHa72}A.T. Amin and S.L. Hakimi, Upper bounds on the order of a
clique of a graph, \emph{SIAM J. Appl. Math.}\textbf{22} (1972), 569-573.

\bibitem {AoHa13}M. Aouchiche and P. Hansen, A survey of Nordhaus-Gaddum type
relations, \emph{Discrete Appl. Math. }\textbf{161} (2013), 466--546.

\bibitem {Bol98}B. Bollob\'{a}s, \emph{Modern Graph Theory}\textit{,} Graduate
Texts in Mathematics, 184, Springer-Verlag, New York (1998), xiv+394 pp.

\bibitem {Csi09}P. Csikv\'{a}ri, On a conjecture of V. Nikiforov, \emph{Disc.
Math.} \textbf{309} (2009), 4522-4526.

\bibitem {HoJo88}R. Horn and C. Johnson, \emph{Matrix Analysis,} Cambridge
University Press, Cambridge, 1985. xiii+561 pp.

\bibitem {Nik06}V. Nikiforov, Linear combinations of graph eigenvalues,
\emph{Electron. J. Linear Algebra }\textbf{15 }(2006), 329-336.

\bibitem {Nik06a}V. Nikiforov, Eigenvalues and degree deviation in graphs,
\emph{Linear Algebra Appl.} \textbf{414} (2006), 347-360

\bibitem {Nik07}V. Nikiforov, Eigenvalue problems of Nordhaus-Gaddum type,
\emph{Discrete Math. }\textbf{307} (2007), 774--780.

\bibitem {NoGa56}E.A. Nordhaus and J. Gaddum, On complementary graphs,
\emph{Amer. Math. Monthly} \textbf{63} (1956), 175--177.

\bibitem {Nos70}E. Nosal, Eigenvalues of Graphs, Master's thesis, University
of Calgary, 1970.

\bibitem {Ter11}T. Terpai, Proof of a conjecture of V. Nikiforov,
\emph{Combinatorica, }\textbf{31 }(2011), 739-754.
\end{thebibliography}
\end{document}